\documentclass[12pt]{article}
\usepackage{amssymb,amsmath}
\usepackage{latexsym,bm}

\usepackage{amsfonts}
\usepackage{latexsym,amsmath,amssymb,amsfonts,epsfig,psfrag,url,graphics,ifpdf,multicol}
\usepackage{eepic,color,colordvi,amscd,amsthm}
\usepackage[section]{algorithm}
\usepackage{algpseudocode}
\usepackage[numbers,sort&compress]{natbib}
\usepackage{indentfirst,graphics,epsfig}
\usepackage{graphicx}
\usepackage{graphics}
\usepackage{tikz}

\newtheorem{theo}{Theorem}
\newtheorem{lem}[theo]{Lemma}

 \theoremstyle{definition}

\theoremstyle{remark}

\def\g{\overrightarrow{g}}
\def\w{\overrightarrow{w}}
\def\p{\overrightarrow{p}}
\def\C{\overrightarrow{C}}
\def\V{\mathcal{V}}
\def\A{\mathcal{A}}
\newcounter{casenum}[theo]

\newcounter{subcasenum}[theo]

\newcounter{claimnum}[theo]

\allowdisplaybreaks
\usepackage{indentfirst,subfig}

\begin{document}
\thispagestyle{plain}

\begin{center} {\Large The stable index of 0-1 matrices
}
\end{center}
\pagestyle{plain}
\begin{center}
{
  {\small Zhibing Chen$^{1}$, Zejun Huang$^{1,}$\footnote{Corresponding author. \\ \indent~~ Email: chenzb@szu.edu.cn (Chen), mathzejun@gmail.com (Huang), mathyanjr@stu.ecnu.edu.cn (Yan) }, Jingru Yan$^{2}$}\\[3mm]
  {\small 1. College of Mathematics and Statistics, Shenzhen University, Shenzhen 518060, China }\\
  {\small 2. Department of Mathematics, East China Normal University, Shanghai 200241,  China }\\

}
\end{center}

\begin{center}

\begin{minipage}{140mm}
\begin{center}
{\bf Abstract}
\end{center}
{\small   We introduce the concept of stable index for 0-1 matrices. Let $A$ be a 0-1 square matrix. If $A^k$ is a 0-1 matrix for every positive integer $k$, then the stable index of $A$ is defined to be infinity; otherwise, the stable index of $A$ is defined to be the smallest positive integer $k$ such that $A^{k+1}$ is not a 0-1 matrix.
We determine the maximum finite stable index of all 0-1 matrices of order $n$ as well as the matrices attaining the maximum finite stable index.

 }

\end{minipage}
\end{center}

\section{Introduction}
Properties on the power of nonnegative matrices attract a lot of attentions in combinatorial matrix theory. Let $A$ be a  nonnegative square matrix. The Perron-Frobenius Theorem states that its spectral radius $\rho(A)$ is an eigenvalue of $A$. If $A$ has no other eigenvalue of modulus $\rho(A)$, then it is said to be {\it primitive}.   Frobenius proved that  $A$ is primitive if and only if there is a positive integer $k$ such that $A^k$ is a positive matrix  (see \cite{XZ}, p.134). Given a primitive matrix $A$, the smallest integer $k$ such that $A^k$ is positive is called the {\it exponent} of $A$.   Wielandt \cite{HW} shows that the exponent of an $n\times n$ primitive matrix is bounded by $(n-1)^2+1$. Dulmage and Mendelsohn \cite{DM} revealed that the exponent set of $n\times n$ primitive matrices is not the set $[1,(n-1)^2+1]$, i.e., there are gaps. Lewin, Vitek, Shao and Zhang \cite{LV,JS1,KZ} determined all the possible exponents of primitive matrices of order $n$. Brualdi and Ross \cite{BR,JR}, Holladay and Varga \cite{HV}, Lewin \cite{ML}, Shao \cite{JS,JS3} investigated exponents of special primitive matrices.
Heap and Lynn \cite{HL1,HL2} studied {\it periods} and the {\it indices of convergence} of  nonnegative matrices, which are defined based on combinatorial properties of powers of nonnegative matrices and have connections with   stochastic theory.

 Denote by $M_n\{0,1\}$ the set of 0-1 matrices of order $n$.
Similar with the   exponent of primitive matrices, we introduce a new parameter on 0-1 matrices as follows, which is called the {\it stable index}.

 {\bf Definition.} {\it Let $A\in M_n\{0,1\}$. If $A^k\in M_n\{0,1\}$ for every positive integer $k$, then the stable index of $A$ is defined to be $\infty$; otherwise, the stable index of $A$ is defined to be the smallest positive integer $k$ such that $A^{k+1}\not\in M_n\{0,1\}$. We  denote the stable index of $A$ by $\theta(A)$.}

  Note that if $\theta(A)$ is finite, then  it is the largest integer $k$ such that $A, A^2,\ldots,A^k$ are all 0-1 matrices.

It is obvious that the stable index can provide an upper bound on the  spectral radius of a 0-1 matrix. Recall that the spectral radius of a 0-1 square matrix does not exceed its maximum row sum (see \cite{XZ}, p.126). Given a matrix $A\in M_n\{0,1\}$ with stable index $k$, since $A^k\in M_n\{0,1\}$, we have
$\rho^k(A)=\rho(A^k)\le n,$
which leads to $$\rho(A)\le  n^{1/k}.$$

In some sense, both the exponent of primitive matrices and the stable index of 0-1 matrices represent the density of nonzero entries in these matrices. Generally, a  matrix with large exponent or stable index can not have many nonzero entries.

Huang and Zhan \cite{HZ} studied 0-1 matrices with infinite stable index. They determined the maximum number of nonzero entries in matrices from $M_n\{0,1\}$ with infinite stable index as well as the matrices attaining this maximum number.
We solve the following problem in this paper.

{\bf Problem 1.} {\it
Given a positive integer $n$. Determine the maximum finite stable index of all 0-1 matrices of order $n$ as well as the 0-1 matrices attaining the maximum finite stable index.}

\section{Main results}
We need some terminology on digraphs.
For a digraph $D$, we denote by $\V(D)$ its vertex set,   $\A(D)$ its arc set,  $(u,v)$ or $uv$ the arc from $u$ to $v$. A sequence of
consecutive arcs $(v_1,v_2),(v_2,v_3),\ldots,(v_{t-1},v_t)$ is called a {\it directed walk} (or {\it walk}) from $v_1$ to $v_t$, which is also written as $v_1v_2\cdots v_t$. A {\it directed cycle} (or {\it cycle}) is a closed walk $v_1v_2\cdots v_tv_1$ with $v_1,v_2,\ldots,v_t$ being distinct. A {\it directed path} (or {\it path}) is a walk in which all the vertices are distinct. The {\it length} of a walk (path) is the number of arcs in the walk. A walk (path) of length $k$ is called a {\it $k$-walk ($k$-path)}. In a digraph $D$, if there is a walk from $u$ to $v$ for all  $u,v\in\V(D)$, then $D$ is said to be {\it strongly connected}.

 Two digraphs $D$ and $H$ are {\it isomorphic}, written $D\cong H$, if there are bijections $\phi: \V (D) \rightarrow \V (H)$ and $\varphi: \A(D) \rightarrow \A(H)$ such that $(u,v)\in \A(D)$ if and only if $(\phi(u),\phi(v))\in \A(H)$. In other words,  $D\cong H$ if and only if we can get $D$ by relabeling the vertices of $H$. We say a digraph $D$ contains a {\it copy} of $H$ if $D$ has a subgraph isomorphic to $H$.

 Denote by $\overrightarrow{C_k}$ the directed cycle with $k$ vertices.    Given an integer $k\ge 2$ and two disjoint   directed cycles $\overrightarrow{C_p}$ and  $\overrightarrow{C_q}$, let $\overrightarrow{g}(p,k,q)$ be the following digraph  obtained by adding a $(k-1)$-path from a vertex of $\overrightarrow{C_p}$ to a vertex of $\overrightarrow{C_q}$. Especially, when $k=2$, we call $\overrightarrow{g}(p,k,q)$ a {\it glasses digraph} and write it simply as $\overrightarrow{g}(p,q)$. It is clear that $|\mathcal{V}(\overrightarrow{g}(p,k,q))|=p+q+k-2.$

 
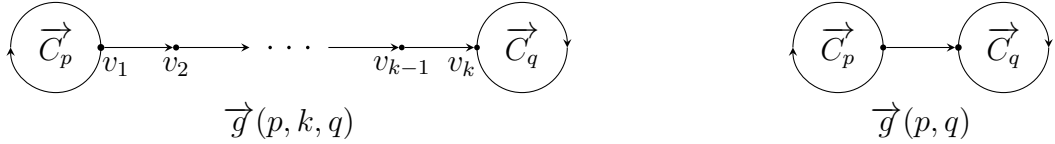
\begin{center}
   \begin{tikzpicture}[>=stealth]
\draw[->](10.4,0)arc[start angle =180,end angle=-180,radius=0.6cm];
\draw(11,0)node[scale=1]{$\overrightarrow{C_{p}}$};
\node[shape=circle,fill=black,scale=0.12](a)at (11.6,0){x};
\node[shape=circle,fill=black,scale=0.12](b) at (12.6,0){y};
\draw[->](a)--(b);
\draw[->](13.8,0)arc [start angle =360,end angle=0,radius=0.6cm];
\draw(13.2,0)node[scale=1]{$\overrightarrow{C_{q}}$};
\draw(12.1,-1)node[scale=1]{$\overrightarrow{g}(p,q)$};

\draw[->](0,0)arc[start angle =180,end angle=-180,radius=0.6cm];
\draw(0.6,0)node[scale=1]{$\overrightarrow{C_{p}}$};
\node[shape=circle,fill=black,scale=0.12](a)at (1.2,0){m};
\draw(1.4,-0.25)node[scale=1]{$v_{1}$};
\node[shape=circle,fill=black,scale=0.12](b) at (2.2,0){n};
\draw(2.2,-0.25)node[scale=1]{$v_{2}$};
\node[shape=circle,fill=white,scale=0.12](c)at (3.2,0){};
\node[shape=circle,fill=black,scale=0.12](d) at (3.45,0){};
\node[shape=circle,fill=black,scale=0.12](e) at (3.7,0){};
\node[shape=circle,fill=black,scale=0.12](f) at (3.95,0){};
\node[shape=circle,fill=white,scale=0.12](g)at (4.2,0){};

\node[shape=circle,fill=black,scale=0.12](h)at (5.2,0){r};
\draw(5.2,-0.25)node[scale=1]{$v_{k-1}$};
\node[shape=circle,fill=black,scale=0.12](i)at (6.2,0){s};
\draw(6,-0.25)node[scale=1]{$v_{k}$};
\draw[->](a)--(b);
\draw[->](b)--(c);
\draw[->](g)--(h);
\draw[->](h)--(i);
\draw[->](7.4,0)arc [start angle =360,end angle=0,radius=0.6cm];
\draw(6.8,0)node[scale=1]{$\overrightarrow{C_{q}}$};
\draw(3.7,-1)node[scale=1]{$\overrightarrow{g}(p,k,q)$};
\end{tikzpicture}
 \end{center}

Given a digraph $D=(\mathcal{V},\mathcal{A})$ with vertex set $\mathcal{V}=\{v_1,v_2,\ldots,v_n\}$ and arc set $\mathcal{A}$, its {\it adjacency matrix} $A_D=(a_{ij})$ is defined by
\begin{equation*}
a_{ij}=\left\{\begin{array}{ll}
1,&\textrm{if } (v_i,v_j)\in \mathcal{A};\\
0,&\textrm{otherwise}.\end{array}\right.
\end{equation*}
We call $\theta(A_D)$ the {\it stable index} of $D$, denoted $\theta(D)$.

 Let $A=(a_{ij})\in M_n\{0,1\}$. We   define its digraph as $D(A)=(\mathcal{V},\mathcal{A})$  with $\mathcal{V}=\{v_1,v_2,\ldots,v_n\}$ and  $\mathcal{A}=\{(v_i,v_j): a_{ij}=1, 1\le i,j\le n\}$.   For convenience, we will always assume the vertex set of $D(A)$  to be $\{1,2,\ldots,n\}$.

Denote by $A(i,j)$ the $(i,j)$-entry of $A$. Given $A\in M_n\{0,1\}$, $A^k(i,j)=t$ if and only if $D(A)$ has exactly $t$ distinct walks of length $k$ from $i$ to $j$. It follows that $D(A)$ contains at most one walk of length $k$ from $i$ to $j$ for all $k\le \theta(A)$ and  $1\le i,j\le n$. If $\theta(A)<\infty$, then there exist vertices $i,j$ such that $D(A)$ contains two distinct walks of length $\theta(A)+1$ from $i$ to $j$.

Let $n$ be a positive integer.  Denote by
$$g(n)=\begin{cases}
           \frac{n^2-1}{4},&\text{if $n$ is odd},\\
           \frac{n^2-4}{4},&\text{if  $n\equiv0$ (mod 4)},\\
           \frac{n^2-16}{4},&\text{if  $n\equiv2$ (mod 4).}
\end{cases}$$
Our main result states as follows.

\begin{theo}\label{th1}
Let $n\geq7$ be an integer and $A\in M_n\{0,1\}$ with $\theta(A)<\infty$. Then
\begin{equation}\label{eq1}
\theta(A)\le g(n)
\end{equation}
with equality if and only if one of the following holds.
\begin{itemize}
\item[(1)] $n\ne 10$ and  $D(A)\cong\g(p,q)$ with
\begin{equation}\label{eq4}
 \{p,q\}=\begin{cases}
           \{\frac{n+1}{2},\frac{n-1}{2}\},&\text{if $n$ is odd,}\\
           \{\frac{n+2}{2},\frac{n-2}{2}\},&\text{if  $n\equiv0$ (mod 4)},\\
           \{\frac{n+4}{2},\frac{n-4}{2}\},&\text{if  $n\equiv2$ (mod 4)}.
\end{cases}
\end{equation}
  \item[(2)] $n=10$ and $D(A)\cong D$  with
  $$ D\in \{\g(3,7),\g(7,3),\g(4,3,5),\g(5,3,4)\}.$$
  \end{itemize}
\end{theo}

Let $E_{ij}$ be the 0-1 matrix with exactly one nonzero entry lying on the  $(i,j)$ position, while its size will be clear from the context. Then $C_n=\sum_{i=1}^{n-1}E_{i,i+1}+E_{n1}$ is the basic circulant matrix of order $n$.

  It is clear that $D(A)\cong\g(p,q)$ if and only if $A$ is permutation similar to
$$\left(\begin{array}{cc}
C_p&E_{ij}\\
0&C_q
\end{array}\right),$$ where $E_{ij}$ is an arbitrary $p\times q$ 0-1 matrix with exactly one nonzero entry. $D(A)\cong\g(p,3,q)$ if and only if $A$ is permutation similar to $$\left(\begin{array}{ccc}
                        C_p&u&0\\
                        0&0&v^T\\
                        0&0&C_q
                        \end{array}\right),$$
where $u,v$ are 0-1 column vectors with exactly one nonzero entry.

{\it Remark.}~~~~Denote by
$s(n)$ the maximum finite stable index of all matrices in $M_n\{0,1\},$ i.e.,$$s(n)=max\{\theta(A):A\in M_n\{0,1\}, \theta(A)<\infty\}.$$ Theorem \ref{th1} shows  $$s(n)=g(n)\quad\text{for}\quad n\ge 7.$$ When $n\le 6$, $s(n)=g(n)$ may not hold. In fact, by using MATLAB we can get all $s(n)$ for $n\le 6$ as follows.
\begin{center}\begin{tabular}{|c|l|l|l|l|l|l|}
  \hline
  $n$ & 2 & 3 & 4 & 5 & 6 \\
  \hline
 $ s(n)$ & 1 & 3 & 4 & 6 & 7 \\
  \hline
\end{tabular} \end{center}

\section{Proof of Theorem \ref{th1}}

In this section we present the proof of Theorem \ref{th1}.
We write $B\le A$ or $A\ge B$ to mean that $A-B$ is a nonnegative matrix. We need the following lemmas.

\begin{lem}
\label{lemma1} Let  $A, B$ be square 0-1 matrices. If $B$ is a principal submatrix of $A$ or $B\le A$, then $$\theta(B)\geq\theta(A).$$
\end{lem}
\begin{proof}
It is obvious.
\end{proof}

\begin{lem}
\label{lemma8}
Let $A\in M_n\{0,1\}$ be irreducible with $n\ge 2$. If $D(A)\cong \C_n$, then   $\theta(A)=
\infty$; otherwise, we have
$\theta(A)\le n$.

\end{lem}
\begin{proof}
Since $A$ is irreducible,  $D(A)$ is strongly connected. If $D(A)\cong\C_n$, then we have $\theta(A)=\infty$. Otherwise, $D(A)$ has two directed cycles $\C_k$ and $\C_t$ with nonempty intersection, which can be written as $\C_k=v_1v_2\cdots v_sv_{s+1}\cdots v_kv_1$ and $\C_t=v_1v_2\cdots v_su_{s+1}\cdots u_tv_1$ with $k\ge t\ge s$ and $$\{v_{s+1},\ldots,v_k\}\cap\{u_{s+1},\ldots,u_t\}=\emptyset.$$
It follows that $D(A)$ has the following distinct walks of length $k+(t-s)+1$ from $v_s$ to $v_1$:
$$\begin{cases}
v_s v_{s+1}\cdots v_k
v_1 \cdots v_s u_{s+1}  \cdots  u_t  v_1,\\
v_s u_{s+1}  \cdots  u_t
v_1 \cdots v_sv_{s+1}\cdots v_k  v_1.
\end{cases}$$
Hence, we have
$\theta(A)\le k+(t-s)\le n.$
\end{proof}
 Denote by $J_{m\times n}$  the $m\times n$ matrix with all entries equal to 1. We have
\begin{lem}
\label{lemma3}Let $m$ and $n$ be relatively prime positive integers. Then for any $1\le i\le m$ and $1\le j\le n$, we have
\begin{equation*}
\sum_{k=0}^{mn-1}{C_m}^kE_{ij}{C_n}^{mn-k-1}=J_{m\times n}.
\end{equation*}
\end{lem}
\begin{proof}
Denote by $P=C_m$, $Q=C_n$.
Note that  ${P}^{k}E_{ij}{Q}^{mn-k-1}$ has exactly one nonzero entry. It suffices to prove
\begin{equation}\label{eq2}
P^{k}E_{ij}Q^{mn-k-1}\neq P^{t}E_{ij}Q^{mn-t-1} \text{\quad for\quad } k,t\in \{0,1,\ldots,mn-1\}~~ \text{such that}~~ k\neq t.
\end{equation}
  Suppose  $P^{k}E_{ij}Q^{mn-k-1}=P^{t}E_{ij}Q^{mn-t-1}$. Then we have
$$P^{k-t}E_{ij}=E_{ij}Q^{k-t},$$
 which implies  $k-t=um=vn$ for some integers $u,v$. Since $m$ and $n$ are relatively prime, we obtain $k-t=rmn$ for some integer $r$. Now $k-t\in[0,mn-1]$ leads to  $k-t=0$. Therefore, we have (\ref{eq2}).
\end{proof}

  Given two positive integers $p$ and $q$, we write $(p,q)=1$ if they are relatively prime. Denote by $lcm(p,q)$ the least common multiple of   $p$ and $q$. Since we can find two distinct walks of length $lcm(p,q)+k-1$ with the same initial and terminal vertices in $\g(p,k,q)$, we have
  \begin{equation}\label{eq3}
  \theta(\overrightarrow{g}(p,k,q))\le lcm(p,q)+k-2.
  \end{equation}

 Note that
\begin{equation}\label{eq6}
\begin{cases}
           (\frac{n+1}{2},\frac{n-1}{2})=1,&\text{if $n$ is odd,}\\
           (\frac{n+2}{2},\frac{n-2}{2})=1,&\text{if  $n\equiv0$ (mod 4)},\\
           (\frac{n+4}{2},\frac{n-4}{2})=1,&\text{if  $n\equiv2$ (mod 4)}.
\end{cases}
\end{equation}
 We have $$g(n)=\max\{pq: p+q=n \text{ and } p,q \text{ are relatively prime}\}.$$
\begin{lem}
\label{lemma4}Let $n\geq7$ be an integer. Then for any matrix $B\in M_n\{0,1\}$ with the form  $$B=\left(\begin{array}{cc}C_p&X\\0&C_q\end{array}\right),\quad X\ne 0,$$   we have
\[\theta(B)\leq g(n).\]
Equality holds if and only if $D(B)\cong\overrightarrow{g}(p,q)$ with $p,q$ satisfying (\ref{eq4}).
\end{lem}
\begin{proof}
Note that  $D(B)$ contains a copy of $\g(p,q)$.
By (\ref{eq3}) and Lemma \ref{lemma1} we have
\begin{equation}\label{eq5}
\theta(B)\leq \theta(\g(p,q))\le lcm(p,q)\leq g(n).
\end{equation}
Now suppose $\theta(B)=g(n)$. Then (\ref{eq5}) implies that  $p,q$ satisfy (\ref{eq4}) and they are relatively prime.
By direct computation we have
$$B^{m}=\left(\begin{matrix}{C_p}^{m}&
\sum_{k=0}^{m-1}{C_p}^kX{C_q}^{m-1-k}\\0&{C_q}^{m}\\
\end{matrix}\right).$$
Next we distinguish two cases.

{\it Case 1.} $X$ has at least two nonzero entries.  Applying Lemma \ref{lemma3} we have
$$\sum_{k=0}^{pq-1}{C_p}^kX{C_q}^{pq-1-k}\ge 2J_{p\times q},$$
which leads to $B^{pq}\not\in M_n\{0,1\}$ and $$\theta(B)<pq=g(n).$$

{\it Case 2.} $X$ has exactly one nonzero entry.  By Lemma \ref{lemma3} and its proof,  we have $$\sum_{k=0}^{m-1}{C_p}^kX{C_q}^{m-1-k}\le J_{p\times q}\quad\text{for}\quad m\le pq, $$
 where equality holds if and only if $m=pq$. On the other hand, it is easy to check that $B^{pq+1}\notin M_n\{0,1\}$. Therefore, we have $$\theta(B)=pq=g(n).$$
Thus we get $D(B)\cong\g(p,q)$ and the second part of the lemma holds.
\end{proof}

\begin{lem}
\label{lemma5} Let $\phi(t)=g(t)-t$. Then
$$\phi(r)\ge\phi(s) \text{ for all positive integers } r\geq7 \text{ and } s<r.$$   Equality holds if and only if $(r,s)=(10,9)$.
\end{lem}
\begin{proof}
For $r\ge 11$ and $s<r$, we have
\begin{eqnarray*}
\phi(r)-\phi(s)\ge \frac{r^2-16}{4}-r-\frac{s^2-1}{4}+s
=\frac{(r+s-4)(r-s)}{4}-\frac{15}{4}>0.
\end{eqnarray*}
Note that
\begin{eqnarray*}
\begin{array}{lllll}
\phi(1)=-1, &\phi(2)=-5, &\phi(3)=-1,& \phi(4)=-1,& \phi(5)=1,\\
\phi(6)=-1, &\phi(7)=5,& \phi(8)=7, &\phi(9)=11, &\phi(10)=11.
\end{array}
\end{eqnarray*}
The conclusion follows clearly.
\end{proof}
\begin{lem}
\label{lemma9}
Let $n\geq7$ be an integer and
$$\mathcal{B}=\big\{B\in M_n\{0,1\}: D(B) \text{ contains a copy of } \g(p,k,q) \text{   for some positive integers }p,q,k\big\}.$$
Then $$\theta(B)\leq g(n)\quad \text{for all}\quad  B\in\mathcal{B}.$$
Moreover, $\theta(B)= g(n)$ if and only if   $D(B)\cong\overrightarrow{g}(p,q)$ with $p,q$ satisfying (\ref{eq4}) for $n\ne 10$  and $$D(B)\cong D~~\text{with}~~ D\in\{\g(3,7),\g(7,3),\g(4,3,5),\g(5,3,4)\}\quad for\quad n=10.$$
\end{lem}

\begin{proof}
Let $B\in\mathcal{B}$.
 Suppose $D(B)$ contains a copy of $\g(p,k,q)$ with $p+q=m$. Then by Lemma \ref{lemma1} and Lemma \ref{lemma5} we have
 \begin{eqnarray}\label{peq2}
 \theta(B)\le \theta(\g(p,k,q))\le lcm(p,q)+k-2\le g(m)+k-2\le g(m)+n-m\le g(n).
 \end{eqnarray}

Now suppose $\theta(B)=g(n)$.  We distinguish two cases.

{\it Case 1. } $m=n$.  The condition $p+q=m$ implies that $B$ is permutation similar with
 $$H=\left(\begin{array}{cc}
 C_p+B_{11}&X\\Y&C_q+B_{22}
 \end{array}
 \right), $$
  where $X\ne 0$. Note that if two 0-1 matrices  are permutation similar, then they have the same stable index. Applying Lemma \ref{lemma8} we get $Y=0$, $B_{11}= 0$ and $B_{22}=0$, since otherwise we would have $\theta(B)=\theta(H)\le n<g(n)$.
 By Lemma \ref{lemma4} we see that
 $\theta(B)= g(n)$ if and only if $D(B)\cong\overrightarrow{g}(p,q)$ with $p,q$ satisfying (\ref{eq4}).

{\it Case 2. } $m<n$.   By (\ref{peq2}) we get
 $$g(m)-m=g(n)-n.$$
 Applying  Lemma \ref{lemma5}  we have $$n=10,\quad m=9, \quad k=3\quad\text{and}\quad lcm(p,q)=g(9),$$
 which leads to $\{p,q\}=\{4,5\}.$  Moreover, $D(B)$ contains no copy of $\g(4,5)$ or $\g(5,4)$. Otherwise we have
   $$\theta(B)\le g(9)<g(10).$$
   Therefore, $D(B)$ contains a copy of $\g(4,3,5)$ or $\g(5,3,4)$.

 Next we prove that if $D(B)$ has a copy of $\g(4,3,5)$, then $D(B)\cong\g(4,3,5)$. Note that $B$ is permutation similar with
 $$H=\left(\begin{array}{ccc}
 C_4+B_{11}&u&B_{13}\\x^T&\alpha&v^T\\B_{31}&y&C_5+B_{33}
 \end{array}\right),$$
 where both $u$ and $v$ are nonzero column vectors.
Applying Lemma \ref{lemma8} we obtain
$$B_{11}=0,\quad B_{33}=0,\quad  x=0\quad \text{and} \quad y=0,$$
since otherwise we would have $\theta(B)=\theta(H)\le 10$.  Since $D(B)$ has no copy of $\g(4,5)$ or $\g(5,4)$,
 we have $B_{13}=B_{31}^T=0$. Moreover, it is easy to check that if $\alpha\ne 0$, then $H^5\not\in M_{10}\{0,1\}$, which leads to $\theta(B)=\theta(H)\le 4$. Therefore,
 \begin{equation}\label{peq4}
 H=\left(\begin{array}{ccc}
 C_4&u&0\\0&0&v^T\\0&0&C_5
 \end{array}\right).
 \end{equation}

By direct computation, we have
$$H^m=\left(\begin{array}{ccc}
{C_4}^m&{C_4}^{m-1}u&\sum_{i=0}^{m-2}{C_4}^i(uv^T){C_5}^{m-2-i}\\
0&0&v^T{C_5}^{m-1}\\
0&0&{C_5}^{m}
\end{array}\right).$$
If $u$ or $v$ has more than one nonzero entry, then applying Lemma \ref{lemma3} we have
$$F(m)\equiv\sum_{i=0}^{m-2}{C_4}^i(uv^T){C_5}^{m-2-i}\ge 2J_{4\times 5}\quad \text{when}\quad m=21,$$
which implies $H^{21}\not\in M_{10}\{0,1\}$ and $$\theta(B)=\theta(H)\le 20< g(10).$$
Hence, each of $u$ and $v$ has exactly one nonzero entry and $D(B)\cong\g(4,3,5)$.

Similarly,  if $D(B)$ has a copy of $\g(5,3,4)$, then $D(B)\cong\g(5,3,4)$.

On the other hand, if $D(B)=\g(4,3,5)$ or  $\g(5,3,4)$, then $B$ or $B^T$ is permutation similar to a matrix $H$ with form (\ref{peq4}), where
  each of $u$ and $v$ has exactly one nonzero entry. It follows that
$$F(m)\le J_{4\times 5}\quad \text{for all}\quad 1\le m\le 21$$
and $F(22)\not\in M_{10}\{0,1\}$, which means $\theta(B)=\theta(H)=21=g(10).$

This completes the proof.
 \end{proof}

\par
Now we are ready to present the proof of Theorem \ref{th1}.
\begin{proof}[Proof of Theorem \ref{th1}]
Let $A\in M_n\{0,1\}$ with $\theta(A)<\infty$.
 Without loss of generality, we may assume
\begin{eqnarray}\label{neq1}
A=\left(\begin{array}{cccc}
A_1&*&\cdots&*\\
0&A_2&\cdots&*\\
\vdots&\vdots&\ddots&\vdots\\
0&0&\cdots&A_k
\end{array}\right)
\end{eqnarray}
with each $A_i$ being an $n_i\times n_i$ irreducible square matrix for $i=1,2,\ldots,k$. We denote by $A_{ij}$ the $(i,j)$-block of $A$ in (\ref{neq1}).

 If there exists $i\in\{1,2,\ldots,k\}$ such that $n_{i}\geq2$ and $D(A_i)$ is not isomorphic to $C_{n_i}$, then by Lemma \ref{lemma8} we  have \begin{equation}\label{peq1}
\theta(A)\leq n<g(n).
\end{equation}
If there  exist $i,j\in  \{1,2,\ldots, k\}$ such that $ n_{i}\geq2,n_{j}\geq2$ and $A_{ij}\neq0$, then $D(A)$ contains a  copy of $\g(n_{i},n_{j})$. Applying Lemma
\ref{lemma4}  we have
$$\theta(A)\leq g(n_{i}+n_{j})\leq g(n).$$
Moreover, $\theta(A)=g(n)$ if and only if   $D(A)\cong \g(p,q)$ with $\{p,q\}=\{n_{i},n_{j}\}$ satisfying (\ref{eq4}).

Now  we are left to verify the case
$$A_i=C_{n_i}~~\text{ for all }~~n_i\ge 2\quad\text{ and }\quad A_{i_{1}i_{2}}=0 ~~\text{ for all }~~ n_{i_1}\ge 2,n_{i_2}\ge2.$$
We consider the digraph $D(A)$. By the definition of stable index, $D(A)$ has two distinct walks $\w_1$ and $\w_2$ with length $\theta(A)+1$ from $x$ to $y$ for some vertices $x,y$. Denote by $\w_1\cup \w_2$ the union of $\w_1$ and $\w_2$, i.e.,  the digraph with vertex set $\V(\w_1)\cup \V(\w_2)$  and arc set $\A(\w_1)\cup\A(\w_2)$.
We distinguish three cases.

 {\it Case 1.}  $\w_1\cup \w_2$ is acyclic. Then both $\w_1$ and $\w_2$ are directed path and $\theta(A)\leq n-1< g(n)$.

{\it Case 2.}  $\w_1\cup \w_2$ has exactly one cycle $\C_p$. If only one of the two walks, say $\w_1$, contains copies of $\C_p$, then $\w_2$ is a directed path with length less than $n-1$, which implies $\theta(A)\leq n-1$.
If both $\w_1$ and $\w_2$ contain copies of $\C_p$, then by deleting a copy of $\C_p$ in each of $\w_1$ and $\w_2$ we get two distinct walks of length $\theta(A)+1-p$ from $x$ to $y$, which contradicts  the definition of $\theta(A)$.
Hence, we always have $\theta(A)<g(n)$ in this case.

{\it Case 3.} $\w_1\cup \w_2$ has at least two cycles.  We distinguish two subcases.

{\it Subcase 3.1.}
 $\w_1\cup \w_2$ contains a copy of $\g(p,k,q)$ for some positive integers $p,q,k$.  Then $(p,q)=(n_i,n_j)$ for some $i,j\in \{1,2,\dots,k\}$.
  Since $A_{ij}=0$ for all $n_i\ge 2, n_j\ge2$, applying Lemma \ref{lemma9} we have  either
  $$\theta(A)\le \theta(\w_1\cup \w_2)<g(n)$$
  or $g(A)=g(n)$ with $$D(A)\cong \g(p,3,q),\quad\text{where}\quad \{p,q\}=\{4,5\}.$$

 {\it Subcase 3.2.} $\w_1\cup \w_2$  does not contain any copy of $\g(p,k,q)$. Then  $\w_1$ and $\w_2$ contain disjoint  cycles $\C_p$ and $\C_q$, respectively. Moreover, $\w_1\cup \w_2$ has the following diagram.
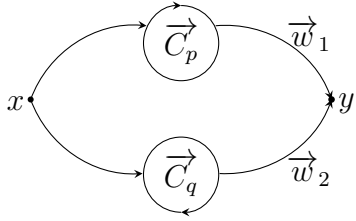
\begin{center}
 \begin{tikzpicture}[>=stealth]
\draw[->](1.5,0.5)arc[start angle =90,end angle=-270,radius=0.5cm];
\draw(1.5,0)node[scale=1]{$\overrightarrow{C_{p}}$};
\node[shape=circle,fill=black,scale=0.12](a)at (-0.5,-0.75){x};
\draw(-0.7,-0.8)node[scale=1]{$x$};
\draw[->](-0.5,-0.75)arc[start angle =160,end angle=85,radius=1.5cm];
\draw[->](-0.5,-0.75)arc[start angle =-160,end angle=-87,radius=1.5cm];
\draw(3.2,0.1)node[scale=1]{$\overrightarrow{w}_{1}$};
\draw[<-](3.5,-0.75)arc[start angle =20,end angle=95,radius=1.5cm];
\draw[<-](3.5,-0.75)arc[start angle =-20,end angle=-93,radius=1.5cm];
\draw(3.2,-1.7)node[scale=1]{$\overrightarrow{w}_{2}$};
\node[shape=circle,fill=black,scale=0.12](b) at (3.5,-0.75){y};
\draw(3.7,-0.8)node[scale=1]{$y$};
\draw[->](1.5,-2.25)arc [start angle =-90,end angle=-450,radius=0.5cm];
\draw(1.5,-1.75)node[scale=1]{$\overrightarrow{C_{q}}$};
\end{tikzpicture}
\end{center}

 Note that there exists a unique directed path from $x$ to $y$ in each of $\w_1$ and $\w_2$, denoted $\p_1$ and $\p_2$, which can be obtained by deleting copies of $\C_p$ and $\C_q$ in $\w_1$ and $\w_2$, respectively. Suppose the length of $\p_1$ and $\p_2$ are  $r$ and $s$, respectively. Then $\w_1$ is the union of $\p_1$ and $(\theta(A)+1-r)/p$ copies of $\C_p$; $\w_2$ is the union of $\p_2$ and $(\theta(A)+1-s)/q$ copies of $\C_q$.

 Let $u$ and $v$ be the smallest nonnegative integers such that
 $r+up= s+vq$.  Then
 $$u=(\theta(A)+1-r)/p\quad\text{and}\quad v=(\theta(A)+1-s)/q.$$
 Otherwise let $\w_3$ be the union of $\p_1$ and $u$ copies of $\C_p$, and let $\w_4$ be the union of $\p_2$ and $v$ copies of $\C_q$. Then $\w_3$ and $\w_4$ are two distinct walks of length $r+up<\theta(A)+1$ from $x$ to $y$, which contradicts the definition of $\theta(A)$.

Next we prove that
 \begin{equation}\label{eq111}
 r+up\le g(n),
 \end{equation}
 which  implies $$\theta(A)\le r+up-1<g(n).$$

 Note that $r, s\le n-1$.
 If $p=1$, then
 \begin{eqnarray*}
 r+up&\le& \min\{s+mq: m \text{ is a nonnegative integer such that }s+mq\ge n-1\}\\
 &<&n-1+q\le n-1+(n-3)\le g(n).
 \end{eqnarray*}
 Similarly, we have (\ref{eq111}) when $q=1$.
 Suppose $p\ge 2$ and $q\ge 2$. Since the cycle $\C_q$ in $\w_2$ is disjoint with $\w_1$, $\p_1$ has at most $n-2$ vertices and we have $r\le n-3$. Similarly,  $s\le n-3$.
 Let $z=lcm(p,q)$. If $r+up\ge g(n)+1$, then
 \begin{eqnarray*}
 r+up-z\ge g(n)+1-g(p+q)\ge g(n)+1-g(n-2)\ge n-3\ge \max\{r,s\}.
 \end{eqnarray*}
 It follows that
 $$r+(u-z/p)p=r+up-z=s+vq-z=s+(v-z/q)q \quad \text{with}\quad u-z/p\ge0,  v-z/q\ge0,$$ which contradicts the choices of $u$ and $v$.
 Thus we obtain (\ref{eq111}).

 Now we can conclude that $\theta(A)\le g(n)$ with equality if and only if $D(A)\cong D$, where $$D\in \{ \g(p,q): p,q \text{ satisfies } (\ref{eq4})\}\cup\{\g(4,3,5),\g(5,3,4)\}.$$
This completes the proof.
\end{proof}

\section*{Acknowledgement}
This work was supported by Science and Technology Foundation of Shenzhen City (No. JCYJ2019080817421) and a Natural Science Fund of Shenzhen University.
The  authors are grateful to Professor Xingzhi Zhan  for helpful discussions on this topic  and valuable suggestions on this paper.


\begin{thebibliography}{WWW}
\bibitem{BR}R. A. Brualdi,  J. A. Ross,   On the exponent of a primitive, nearly reducible matrix, Math. Oper. Res. 5 (1980) 229-241.
\bibitem{DM} A. L. Dulmage, N. S. Mendelsohn, Gaps in the exponent set of primitive matrices,  Illinois J. Math. 8 (1964) 642-656.
\bibitem{HL1}   B. R.  Heap,  M. S. Lynn,  The Structure of Powers of Nonnegative Matrices I. The Index of Convergence, SIAM J. Appl. Math. 14 (1966) 610-639.
\bibitem{HL2}   B. R.  Heap,  M. S. Lynn, The Structure of Powers of Nonnegative Matrices II. The Index of Maximum Density, SIAM J. Appl. Math. 14 (1966) 762-777.
    \bibitem{HV} J. C. Holladay, R. S. Varga, On powers of nonnegative matrices, Proc. Amer. Math. Soc. 9 (1958) 631-634.
 \bibitem{HZ} Z. Huang, X. Zhan,  Extremal digraphs whose walks with the same initial and terminal vertices have distinct lengths, Discrete Math. 312 (2012) 2203-2213.
     \bibitem{ML} M. Lewin, On exponents of primitive matrices, Numer. Math. 18 (1971/72) 154-161.
\bibitem{LV} M. Lewin, Y. Vitek, A system of gaps in the exponent set of primitive matrices, Illinois J. Math. 25 (1981) 87-98.
    \bibitem{JR}J. A. Ross,  On the exponent of a primitive, nearly reducible matrix. II, SIAM J. Algebraic Discrete Methods 3 (1982)  395-410.
\bibitem{JS1}J.-Y. Shao, On a conjecture about the exponent set of primitive matrices, Linear Algebra Appl. 65 (1985) 91-123.
\bibitem{JS} J.-Y. Shao, The exponent set of symmetric primitive matrices, Scientia Sinica (Ser. A) 30 (1987)  348-358.
\bibitem{JS3}J.-Y. Shao, The exponent set of primitive, nearly reducible matrices, SIAM J. Algebraic Discrete Methods 8 (1987) 578-584.
\bibitem{HW} H. Wielandt, Unzerlegbare nicht negative matrizen, Math. Z. 52 (1950)  642-648.
\bibitem{XZ}X. Zhan, Matrix Theory, Graduate Studies in Mathematics 147, American Mathematical Society, Providence, RI, 2013.
 \bibitem{KZ} K. M. Zhang, On Lewin and Vitek's conjecture about the exponent set of primitive matrices, Linear Algebra Appl. 96 (1987) 101-108.

\end{thebibliography}
\end{document}